\documentclass[12pt]{amsart}

\usepackage{amsmath,amsthm,amsfonts,amscd,amssymb}
\usepackage{times,euler,eucal,eufrak,graphicx}
\usepackage{pst-node}
\usepackage{pst-plot}
\usepackage{fullpage}

\usepackage{hyperref}

\title[CEP for quantum groups]
{The Connes embedding property for quantum group von Neumann algebras} 

\author {Michael Brannan}
\address{Michael Brannan,
Department of Mathematics,
University of Illinois at Urbana-Champaign, 376 Altgeld Hall, 1409 W. Green Street Urbana, IL 61801, USA}
\email{mbrannan@illinois.edu}

\author {Beno\^\i{}t Collins}
\address{Beno\^\i{}t Collins,
Department of Mathematics, Kyoto University,
D\'epartement de Math\'ematique et Statistique, Universit\'e d'Ottawa,
585 King Edward, Ottawa, ON, K1N6N5 Canada,
and
CNRS, Institut Camille Jordan Universit\'e  Lyon 1,
France}
\email{bcollins@uottawa.ca}

\author {Roland Vergnioux}
\address{Roland Vergnioux,
UFR Sciences, LMNO,
Universit\'e de Caen Basse-Normandie,
Esplanade de la Paix, CS 14032, 14032 CAEN cedex 5, France}
\email{roland.vergnioux@unicaen.fr}

\theoremstyle{plain}
\newtheorem{lemma}{Lemma}[section]
\newtheorem{theorem}[lemma]{Theorem}
\newtheorem{proposition}[lemma]{Proposition}
\newtheorem{corollary}[lemma]{Corollary}

\theoremstyle{definition}
\newtheorem{definition}{Definition}

\theoremstyle{remark}
\newtheorem{remark}{Remark}

\DeclareMathOperator{\tr}{tr} 

\renewcommand{\iff}{\textbf{iff} }
 
\newcommand{\tens}{\mathord{\otimes}} 
\newcommand{\hatfree}{\mathbin{\hat\ast}} 
\newcommand{\myop}[1]{\operatorname{#1}}
\newcommand{\Ker}{\myop{Ker}} 

\newcommand{\Irr}{\myop{Irr}}

\newcommand{\Span}{\myop{span}}
\newcommand{\Fix}{\myop{Fix}}

\newcommand{\Hom}{\myop{Hom}}

\newcommand{\Pol}{\myop{Pol}}
\newcommand{\CEP}{\myop{CEP}}

\newcommand{\N}{\mathbb N}
\newcommand{\C}{\mathbb C}

\newcommand{\R}{\mathbb R}
\newcommand{\G}{\mathbb G}
\renewcommand{\H}{\mathbb H}

\newcommand{\TL}{\text{TL}}
\newcommand{\mc}{\mathcal}

\newcommand{\Bb}{\mathcal{B}}
\newcommand{\Cc}{\mathcal{C}} 

\newcommand{\Tt}{\mathcal{T}}

\begin{document}

\begin{abstract}
For a compact quantum group $\G$ of Kac type, we study the existence of a Haar trace-preserving embedding of the von Neumann algebra $L^\infty(\G)$ into an ultrapower of the hyperfinite II$_1$-factor (the Connes embedding property for  $L^\infty(\G)$).  We establish a connection between the Connes embedding property for $L^\infty(\G)$ and the structure of certain quantum subgroups of $\G$, and use this to prove that the II$_1$-factors $L^\infty(O_N^+)$ and $L^\infty(U_N^+)$
  associated to the free orthogonal and free unitary quantum groups have 
  the Connes embedding property for all $N \ge 4$.  As an application, we deduce that the free entropy dimension of the standard generators  of $L^\infty(O_N^+)$ equals $1$ for all $N \ge 4$.  We also mention an application of our work to the problem of classifying the quantum subgroups of $O_N^+$.
\end{abstract}

\keywords{The Connes embedding problem, quantum group, matricial microstates, von Neumann algebra, free entropy.}
\thanks{2010 \it{Mathematics Subject Classification:}
\rm{46L65, 46L54, 20G42, 22D25}}

\maketitle

\section{Introduction}
\label{sec:intro}

The Connes embedding problem asks whether any finite von Neumann algebra with
separable predual embeds into an ultrapower of the hyperfinite II$_1$-factor in
a trace preserving way. This question was raised by Connes in \cite{co1}. See
\cite{CaLu,oza,pes} for nice introductions on this topic.  This central
question in the theory of operator algebras is still open, and has ramifications in many other
areas of mathematics, such as e.g.  non-commutative probability theory, quantum
information, and non-commutative algebraic geometry. In probabilistic terms, this
question amounts to knowing whether any finite family of elements of a bounded
tracial non-commutative probability space admits an asymptotic matrix model.  In the framework of Voiculescu's free entropy theory, this amounts to asking about the existence of
matricial microstates, see \cite{v1, v2}. 

The aim of this paper is to provide a new class of examples of Connes embeddable
von Neumann algebras, namely von Neumann algebras arising from non-coamenable compact
quantum groups of Kac type.  Within the operator algebraic framework, arguably
the most studied examples of compact quantum groups of Kac type include the free
orthogonal quantum groups $O_N^+$ and the free unitary quantum groups $U_N^+$.
Over the last two decades, this class of quantum groups has been extensively
studied, and remarkable connections have emerged between these quantum groups
and free probability theory.  These connections occur at the level of quantum
symmetries and asymptotic freeness results \cite{ba2, BaCo,BrSAF, Cu, CuSp}, and
also at the operator algebra level \cite{BrAP, dCFY, FiVe, Fr13, Is12, VaVe}.
In particular, the von Neumann algebras $L^\infty(O_N^+)$ and $L^\infty(U_N^+)$
share many of the same structural properties with the free group factors: they
are full type II$_1$-factors; they are strongly solid, and in particular they are
prime and have no Cartan subalgebra; they have the Haagerup property and are
weakly amenable with Cowling-Haagerup constant $1$ (CMAP).  But unlike the case of the free group factors, the II$_1$-factors $L^\infty(O_N^+)$ and $L^\infty(U_N^+)$, $N \ge 3$ were not known to be Connes embeddable (i.e., to admit matricial microstates).

Our main result in this paper is that as soon as $N \ge 4$, these von Neumann
algebras have the Connes embedding property. This result is presented as a
corollary of a more general and new stability result of the Connes embedding
property under certain quantum group theoretical operations.  More precisely, we consider a generalization within the category of compact quantum groups, of the notion of a compact group being topologically generated by a pair of closed subgroups (see Definition \ref{qsg}).  Using this technology, we show in Theorem \ref{corollary:CEPsgs} that if a Kac type compact quantum group $\G$ is topologically generated by a pair of quantum subgroups $\G_1,\G_2$ which have the additional property that $L^\infty(\G_1),L^\infty(\G_2)$ are Connes embeddable, then $L^\infty(\G)$ is also Connes embeddable.   The utility of
Theorem \ref{corollary:CEPsgs} lies in the fact that it can be used to reduce the problem of verifying the Connes embedding property for $L^\infty(\G)$ to the (possibly easier) problem of verifying the same property for the ``smaller'' algebras $L^\infty(\G_1),L^\infty(\G_2)$.  When dealing with the specific examples of  $L^\infty(O_N^+)$ and $L^\infty(U_N^+)$ , this theorem in fact allows us to establish the Connes embedding property via an induction procedure over the dimension parameter $N$. 
Another interesting feature of the embedding of $L^\infty(\G)$ into an ultrapower of the hyperfinite II$_1$-factor given by Theorem \ref{corollary:CEPsgs} is that it is obtained somewhat explicitly in terms of the embeddings associated to the given quantum subgroups $\G_1,\G_2$.  See Remark \ref{rem:explicit} for details.   It is our hope that this observation can lead to a more systematic understanding of how to construct explicit matricial microstates for certain quantum group von Neumann algebras.

The paper is organized as follows.  Section \ref{sec:intro} contains preliminaries about
compact and free quantum groups.  Section \ref{section:cep} recalls facts about the Connes
embedding property and relates this property for quantum group von Neumann algebras to the structure of quantum subgroups. In Section \ref{section:mainresult} the
Connes embedding property for $L^\infty(O_N^+)$ and $L^\infty(U_N^+)$, $N\geq 4$, is derived through
the study of specific quantum subgroups. Finally, in Section \ref{section:apps} we consider some
applications of our results to free entropy dimension and to the problem of classifying the quantum subgroups of $O_N^+$ which contain the classical orthogonal group $O_N$ as a quantum subgroup.

\section{Preliminaries}
\label{sec:intro}

\subsection{Compact quantum groups} \label{section:cqg}

In this section we recall some basic facts on compact quantum groups. We follow
\cite{wo2} and \cite{MaVa} and refer to these papers for the facts stated below.

A compact quantum group is a pair $\G = (A,\Delta )$ where $A$ is a unital
$C^*$-algebra and $\Delta : A\to A\otimes A$ is a unital $*$-homomorphism
satisfying
\begin{align*}
  &(\iota \otimes \Delta )\Delta = (\Delta \otimes \iota)\Delta \quad\mbox{   (coassociativity)}\\
  &[\Delta (A)(1\otimes A)]=[\Delta (A) (A\otimes 1)]=A\otimes A \quad \mbox{
    (non-degeneracy)},
\end{align*}
where $[S]$ denotes the norm-closed linear span of a subset $S\subset A \otimes
A$. Here and in the rest of the paper,  the symbol $\otimes$ will denote the minimal tensor product of C$^\ast$-algebras, $\overline{\otimes}$ will denote the spatial tensor product of von Neumann algebras, and $\odot$ will denote the algebraic tensor product of complex associative algebras.  The homomorphism $\Delta$ is called a \textit{coproduct}.  The C$^\ast$-algebra $A$ together with the coproduct $\Delta$ is often called a {\it Woronowicz C$^\ast$-algebra}.


For any compact quantum group $\G = (A,\Delta)$, there exists a unique
\textit{Haar state} $h: A\to \mathbb{C}$ which satisfies the following left and
right invariance property, for all $a\in A$:
\begin{equation} \label{Haar} (h\otimes \iota)\Delta (a)=(\iota\otimes h)\Delta
  (a)=h(a)1.
\end{equation}
We say that a compact quantum group $\G$ is \textit{of Kac type} if $h$ is a
tracial state.  Note that in general $h$ may not be faithful on $A$.  In any
case, we can construct a GNS representation $\pi_h:A \to \mc B(L^2(\G))$, where
$L^2(\G)$ is the Hilbert space obtained by separation and completion of $A$ with
respect to the sesquilinear form $\langle a |b \rangle = h(a^*b),$ and $\pi_h$
is the natural extension to $L^2(\G)$ of the left multiplication action of $A$
on itself. The C$^\ast$-algebra \[C_r(\G) = \pi_h(A) \subset \mc B(L^2(\G))\] is
called the \textit{reduced C$^\ast$-algebra of functions on $\G$}.  Due to the
invariance properties of the Haar state $h$, the coproduct $\Delta$ extends to a
unital $\ast$-homomorphism $\Delta_r:C_r(\G) \to C_r(\G) \otimes C_r(\G)$,
making the pair $(C_r(\G),\Delta_r)$ a compact quantum group (with faithful Haar
state), called the \textit{reduced version of $\G$}.  The \textit{von Neumann
  algebra of $\G$} is given by
\[L^\infty(\G) = C_r(\G)'' \subseteq \mc B(L^2(\G)).\] We note that $\Delta_r$
extends to an injective normal $\ast$-homomorphism $\Delta_r:L^\infty(\G)\to
L^\infty(\G) \overline{\otimes} L^\infty(\G)$, and the Haar state on $C_r(\G)$
extends to a faithful normal $\Delta_r$-invariant state on $L^\infty(\G)$.

Let $H$ be a Hilbert space and $u\in M(\mc K(H)\otimes A)$ be an invertible
(unitary) multiplier. The multiplier $u$ is called a \textit{(unitary)
  representation} of $\G$ if, following the leg numbering convention,
\begin{equation} \label{rep} (\iota \otimes \Delta) u=u_{12}u_{13}.
\end{equation}  
If $\dim H = n < \infty$, then (after fixing an orthonormal basis of $H$) we can
identify $u$ with an invertible matrix $u = [u_{ij}] \in M_n(A)$ and \eqref{rep}
means exactly that
\[
\Delta(u_{ij}) = \sum_{k=1}^n u_{ik} \otimes u_{kj} \qquad (1 \le i,j \le n).
\]
Of course the unit $1 \in A$ is always a representation of $\G$, called the
\textit{trivial representation}.

Let $u \in M(\mc K(H_1)\otimes A)$ and $v \in M(\mc K(H_2)\otimes A)$ be two
representations of $\G$.  Then then their \textit{direct sum} is the
representation $u \oplus v \in M(\mc K(H_1 \oplus H_2)\otimes A)$, and their
\textit{tensor product} is the representation $u\otimes v:= u_{13}v_{23} \in
M(\mc K(H_1 \otimes H_2)\otimes A)$.  An \textit{intertwiner} between $u$ and
$v$ is a bounded linear map $T :H_1 \to H_2$ such that $(T \otimes \iota )u = v
(T \otimes \iota)$.  The Banach space of all such intertwiners is denoted by
$\text{Hom}_\G(u, v)$. If there exists an invertible (unitary) intertwiner
between $u$ and $v$, they are said to be \textit{(unitarily) equivalent}. A
representation is said to be irreducible if its only self-intertwiners are the
scalar multiples of the identity map.  It is known that each irreducible
representation of $\G$ is finite dimensional and every finite dimensional
representation is equivalent to a unitary representation. In addition, every
unitary representation is unitarily equivalent to a direct sum of irreducible
representations.

Denote by $\Irr(\G)$ the collection of equivalence classes of finite dimensional
irreducible unitary representations of $\G$.  For each $\alpha \in \Irr( \G)$,
select a representative unitary representation $u^\alpha = [u^\alpha_{ij}] \in
M_{n_\alpha}(A)$.  The linear subspace $\Pol( \G) \subseteq A$ spanned by
$\{ u_{ij}^\alpha : \alpha \in \Irr(\hat \G), \ 1 \le i,j \le n_\alpha\}$ is a
dense $\ast$-subalgebra of $A$, called \textit{the algebra of polynomial
  functions on $\G$}.  $\Pol(\G)$ is in fact a Hopf $\ast$-algebra with
coproduct $\Delta_0: \Pol(\G) \to \Pol(\G) \odot \Pol(\G)$ given by restriction
of the coproduct $\Delta$.  The \textit{antipode} $S:\Pol(\G) \to
\Pol(\G)^{\text{op}}$ is the automorphism given by $(\iota \otimes S)(u^\alpha)
= (u^\alpha)^*$ and the \textit{counit} is the $\ast$-character
$\epsilon:\Pol(\G) \to \C$ given by $(\iota \otimes \epsilon)(u^\alpha) = 1.$
For any compact quantum group $\G$, the Haar state $h$ is always faithful on
$\Pol(\G)$.  Moreover $\G$ is Kac type if and only if $S^2 = \iota$.

The universal enveloping C$^\ast$-algebra of $\Pol(\G)$ is denoted by $C^u(\G)$.
By universality, the coproduct on $\Pol(\G)$ extends continuously to a coproduct
$\Delta_u$ on $C^u(\G)$, making $(C^u(\G), \Delta_u)$ a compact quantum group
(the \textit{universal version of} $\G$).  A compact quantum group $\G$ is
called \textit{coamenable} if the Haar state is faithful on $C^u(\G)$.  When
$\G$ is of Kac type, this is equivalent to $L^\infty(\G)$ being an injective
finite von Neumann algebra \cite{ruan}.

Given a pair of compact quantum groups $\G, \mathbb H$, we call $\mathbb H$ a 
\textit{quantum subgroup} of $\G$ if there exists a surjective
$\ast$-homomorphism $\pi:C^u(\G) \to C^u(\mathbb H)$ intertwining the respective
coproducts: $\Delta_{u,\mathbb H} \circ \pi = (\pi \otimes \pi ) \circ
\Delta_{u,\G}$. Given two compact quantum groups $\G_1$, $\G_2$ the \textit{dual
  free product} $\G_1 \hatfree \G_2$ of $\G_1$ and $\G_2$ is given by the
reduced free product algebra $A = C_r(\G_1) *_r C_r(\G_2)$ with respect to the
Haar states, endowed with the unique coproduct extending the ones of $C_r(\G_1)$
and $C_r(\G_2)$.

A compact quantum group $\G$ is called a \textit{compact matrix quantum group}
if there exists a finite dimensional unitary representation $u = [u_{ij}] \in
M_n(A)$ whose matrix elements generate $A$ as a C$^\ast$-algebra.  Such a
representation $u$ is called a \textit{fundamental representation} of $\G$.  In
this case $\Pol(\G)$ is simply the $\ast$-algebra generated by $(u_{ij})_{1 \le
  i,j \le n}$.

\begin{remark} \label{rem:discrete} Associated to any compact quantum group $\G$
  one can construct a unique {\it dual discrete quantum group} $\hat \G$.  See
  \cite{BaSk,PoWo} for an introduction to the basic theory of discrete-compact
  quantum group duality. Although we do not use the technology of discrete
  quantum groups here, it is useful to note that through the discrete-compact
  quantum group duality, the algebras $\Pol(\G)$, $C^u(\G)$, $C_r(\G)$,
  $L^\infty(\G)$ introduced before play the role of the familiar algebras $\C[\Gamma]$,
  $C^*(\Gamma)$, $C^*_r(\Gamma)$, $\mc L(\Gamma) = \lambda(\Gamma)''$
  (respectively) associated to a discrete group $\Gamma$. With this terminology,
  we see that the discrete quantum group associated with a dual free product
  $\G_1 \hatfree \G_2$ can be interpreted as the free product quantum group
  $\hat \G_1 * \hat \G_2$, and this justifies our notation.
\end{remark}

\subsection{Free orthogonal and free unitary quantum groups.}

We now introduce the free orthogonal and free unitary quantum groups, which form
the central objects of study in this paper. These quantum groups were first
introduced in the operator algebraic framework by Wang \cite{Wa}.  Purely
algebraic versions of these objects were also introduced by Dubois-Violette and Launer in \cite{DuLa}.

Let $N \ge 2$.  The \emph{free orthogonal quantum group} $O_N^+$ is the compact
quantum group given (in universal form) by the pair $(C^u(O_N^+),\Delta_u)$,
where
\[C^u(O_N^+) = C^*\big((u_{ij})_{1 \le i,j \le N} | \ u = [u_{ij}] \text{ is
  unitary in }M_N(C^u(O_N^+)) \ \& \ \bar u = u \big),\] where $\bar u =
[u_{ij}^*]$.  The coproduct $\Delta_u$ is defined so that $u$ becomes a unitary
representation of $O_N^+$. That is, $\Delta_u (u_{ij})=\sum_{k=1}^N
u_{ik}\otimes u_{kj}$ for each $1 \le i,j \le N$.  Note that the abelianization
of $C_u(O_N^+)$ is naturally isomorphic to the C$^\ast$-algebra of continuous
functions on the compact Lie group $O_N$.  In particular, $O_N$ is a quantum
subgroup of $O_N^+$.

The \textit{free unitary quantum group} $U_N^+$ is defined in the same fashion as
$O_N^+$, except that we no longer assume that the generators of $C^u(U_N^+)$ are
self-adjoint.  More precisely, we define
\[C^u(U_N^+) = C^*\big((u_{ij})_{1 \le i,j \le N} | \ u = [u_{ij}] \ \& \ \bar u
\text{ are unitary}\big).\] Similarly, $U_N$ is a quantum subgroup of $U_N^+$.

From the above definitions, it follows that for $\G = O_N^+, U_N^+$, the
antipode $S:\Pol(\G) \to \Pol(\G)$, $S(u_{ij}) = u_{ji}^*$ satisfies $S^2 =
\iota$.  In particular, the Haar states $h_\G$ are tracial.

\subsection{Invariant theory for $O_N^+$}
\label{sec_inv_theory}

Let $N \ge 2$ and $u$ be the fundamental representation of $O_N^+$ acting on the Hilbert space $H = \C^N$.  In this
section we briefly recall the structure of the intertwiner spaces
$\Hom_{O_N^+}(u^{\otimes k}, u^{\otimes l})$, $k, \l \in \N_0$, as described in
\cite{BaCo} (see also \cite{ba1}).  We start with a couple of definitions.

\begin{definition} Let $k,l \in \N_0$ be such that $k + l \in 2\N_0$. We denote
  $NC_2(k,l)$ the set of \emph{non-crossing pair partitions} of $k$ upper points
  and $l$ lower points, that is, partitions that can be represented by diagrams
  formed by an upper row of $k$ points, a lower row of $l$ points, and $(k +
  l)/2$ non-crossing strings joining pairs of points. The vector space of
  \emph{$(k,l)$ Temperley-Lieb diagrams} is the abstract complex vector space
  $\TL(k,l)$ freely spanned by $NC_2(k,l)$.
\end{definition}

Consider now the Hilbert space $H = \C^N$ and denote by
$(e_i)_{i=1}^N$ its standard basis. Each diagram $p\in NC_2(k,l)$ acts as a
linear map $T_p:H^{\otimes k} \to H^{\otimes l}$ given by
\begin{align}\label{Tp}
  T_p(e_{i_1}\otimes\ldots\otimes e_{i_k})=\sum_{j_1\ldots j_l =
    1}^N\begin{pmatrix}i_1 \ldots i_k\cr p\cr j_1\ldots
    j_l\end{pmatrix}e_{j_1}\otimes\ldots\otimes e_{j_l},
\end{align}
where the middle symbol is $1$ if all strings of $p$ join pairs of equal
indices, and is $0$ if not. We denote $TL_N(k,l) \subseteq \mc B(H^{\otimes
  k},H^{\otimes l})$ the subspace spanned by the maps $T_p$, $p \in
NC_2(k,l)$. This subspace is related to $O_N^+$-intertwiners as follows.

\begin{theorem}[\cite{ba1, BaCo}] \label{thm_invariants} Let $u$ be the
  fundamental representation of $O_N^+$.  Then for all $N \ge 2$,
  \begin{displaymath}
    \Hom_{O_N^+}(u^{\otimes k},u^{\otimes l})=TL_N(k,l).
  \end{displaymath}
  Moreover the family of linear maps $(T_p)_{p \in NC_2(k,l)}$ is linearly
  independent as soon as $N \ge 2$.
\end{theorem}

If the first index is zero we omit it and we denote $NC_2(k) = NC_2(0,k)$,
$TL(k) = TL(0,k)$. When there is no risk of confusion we will denote $\Fix_k =
\Fix(u^{\otimes k}) = \Hom_{O_N^+}(1,u^{\otimes k})\subseteq H^{\otimes k}$ the
subspace of fixed vectors for the representation $u^{\otimes k}$ of $O_N^+$, and
according to Theorem~\ref{thm_invariants} we have for $N\geq 2$:
\begin{displaymath}
  \Fix_k = \Span~ \{T_p \mid p\in NC_2(k)\} = TL_N(k).
\end{displaymath}

We also recall that Theorem~\ref{thm_invariants} has a classical counterpart
dating back to Brauer \cite{Brauer}. We denote $P_2(k,l)$ the set of all pair
partitions of $k$ upper points and $l$ lower points, and we observe that $T_p$
can still be defined for any $p \in P_2(k,l)$. Then we have
\begin{displaymath}
  \Hom_{O_N}(v^{\otimes k},v^{\otimes l}) = \Span~\{T_p \mid p \in P_2(k,l)\},
\end{displaymath}
where $v$ is the fundamental representation of $O_N$ on $H = \C^N$. Note however
that the maps $T_p$, for $p \in P_2(k,l)$, are not linearly independent for any
$N$, as soon as $k+l$ is big enough.

\begin{remark} \label{rem:unitaryinvariants} In what follows, we will not need
  to discuss the details of the invariant theory of $U_N^+$, although it was
  thoroughly described in \cite{ba2, BaCo}.  We will however, use the fact that
  if $w = [w_{ij}]$ and $\bar w = [w_{ij}^*]$ denote the fundamental
  representation of $U_N^+$ and its conjugate, acting on $H = \C^N$ and $\bar H$
  (respectively), then after canonically identifying $H$ and $\bar H$ in the
  obvious way, we have the equality of intertwiner spaces
  \[\Hom_{U_N^+}(1,(w \otimes \bar w)^{\otimes k}) = \Hom_{O_N^+}(1,u^{\otimes
    2k}) \qquad (k \ge 1). \]
\end{remark}

\section{The Connes embedding property} \label{section:cep}

Before specializing to quantum groups, let us first recall a few basic things
about the Connes embedding property in the context of unital $\ast$-algebras.

\subsection{Connes embeddable tracial $\ast$-algebras}

Let $A$ be a unital $\ast$-algebra and $\tau:A \to \C$ a (not necessarily
faithful) tracial state.  $A$ can be endowed with a sesquilinear form $\langle
x,y\rangle:=\tau (x^*y)$.  Let $N_\tau:=\{x\in A: \langle x,x\rangle=0\}$. By
the Cauchy-Schwarz inequality, $N_\tau:=\{x\in A: \langle x,y\rangle=0 \ \forall
y \in A\}$, and therefore $N_\tau$ is a linear subspace of $A$ and $\langle
\cdot ,\cdot \rangle$ can be defined naturally on the quotient space $A/N_\tau$,
where it is a non-degenerate sesquilinear form.  The resulting completion of
$A/N_\tau$ is denoted by $L^2(A,\tau)$.  We denote by $\Lambda_\tau:A\to A/N_\tau\subset L^2(A,\tau)$ the quotient map.  If $A$ is
generated as a $\ast$-algebra by elements $(x_i)_{i \in I}$ such that $\tau
((x_i^*x_i)^n)^{1/n}$ is bounded for each $i$, then there exists a unital
$\ast$-homomorphism $\pi_\tau:A \to B(L^2(A,\tau))$ satisfying
\[\pi_\tau (x)\Lambda_\tau(y)=\Lambda_\tau(xy) \qquad (x,y \in A).\]
The representation $\pi_\tau$ is usually called the GNS representation of $A$
with respect to the tracial state $\tau$.  Taking double commutants, we obtain
from $A$ a von Neumann algebra $\pi_\tau(A)'' \subseteq \mc B(L^2(A,\tau))$, and
the original state $\tau$ extends by continuity to a faithful normal tracial
state on $\pi_\tau(A)''$ still denoted by $\tau$.  Throughout this paper, we will always assume that our tracial $\ast$-algebras $(A,\tau)$ are such that $\pi_\tau$ exists and that the von Neumann algebra $\pi_\tau(A)''$ has a separable predual.

Let us briefly recall the ultrapower construction for the hyperfinite II$_1$-factor.  Let $R$ denote the hyperfinite
II$_1$-factor and $\tau_R$ its unique faithful normal tracial state. Let
$\omega$ be a free ultrafilter on $\mathbb{N}$ and let $I_\omega \subseteq
\ell^\infty(\mathbb{N},R)$ be the ideal consisting of those sequences
$(x_n)_{n=1}^\infty$ such that $\lim_{n\to\omega}\tau_R((x_n)^*x_n)=0$.  Then the {\it ultrapower of the hyperfinite II$_1$-factor (along the ultrafilter $\omega$)} is the quotient $R^\omega:=\ell^\infty(\mathbb{N},R)/I_\omega$, which turns out
to be a II$_1$-factor with faithful normal tracial state
$\tau_{R^\omega}((x_n)_n + I_\omega) = \lim_{n\to\omega}\tau_R(x_n)$.  We now
come to the fundamental concept of this paper.

\begin{definition}\label{def-cep}
  Let $A$ be a unital $\ast$-algebra equipped with a tracial state $\tau$.  The
  state $\tau$ is said to have the {\it Connes embedding property} if the finite
  von Neumann algebra $(\pi_\tau(A)'',\tau)$ can be embedded into an ultrapower
  $R^\omega$ of the hyperfinite II$_1$-factor $R$ in a trace--preserving
  way.  We write $\CEP(A)$ for the set of such tracial states $\tau:A \to \C$.
\end{definition}

Since our point of view and motivation is that of matricial microstates, let us also recall
the following definition.

\begin{definition} \label{defn:microstates} Let $A$ be a unital $\ast$-algebra
  equipped with a faithful tracial state $\tau$.  If $X=(x_1,\ldots,x_n)$ is a
  finite subset of $A_{sa}:=\{x\in A: x^*=x\}$, we say that $X$ {\em has
    matricial microstates (relative to $\tau$)} if for every $m\in\mathbb{N}$
  and every $\epsilon>0$, there is a $k\in\mathbb{N}$ and self-adjoint matrices
  $a_1,\ldots,a_n \in M_k(\C)$ such that whenever $1\le p\le m$ and
  $i_1,\ldots,i_p\in\{1,\ldots,n\}$, we have
  \begin{equation}\label{eq:Amomxmom}
    |\text{tr}_k(a_{i_1}a_{i_2}\cdots a_{i_p})-\tau(x_{i_1}x_{i_2}\cdots x_{i_p})|<\epsilon.
  \end{equation}
  where $\tr_k$ is the normalized trace on $M_k(\mathbb{C})$ satisfying $\tr_k(1) = 1$.
\end{definition}


The following von Neumann algebraic result connecting the existence of matricial
microstates to the Connes embedding property is well known, see for example
\cite[Prop. 3.3]{MR2465797}:
\begin{proposition}\label{prop:microstates}
  Let $M$ be a von Neumann algebra with separable predual equipped with a
  faithful normal tracial state $\tau$. Then the following are equivalent:
  \begin{enumerate}
  \item $\tau \in \CEP(M)$ (i.e., $M$ has the Connes embedding property).
  \item Every finite subset $X\subset M_{sa}$ has matricial microstates
    relative to $\tau$.
  \item If $Y\subset M_{sa}$ is a generating set for $M$, then every finite
    subset $X\subset Y$ has matricial microstates.
  \end{enumerate}
  In particular, if $Y\subset M_{sa}$ is a finite generating set of $M$ then the above
  conditions are equivalent to $Y$ having matricial microstates.
\end{proposition}

The following lemma gives some important stability properties of $\CEP(A)$ that
will be essential in the sequel.

\begin{lemma} \label{lemma:stabilityproperties} Let $(A,\tau)$ and
  $(A_i,\tau_i)_{i=1,2}$ be unital $\ast$-algebras equipped with tracial states
  $\tau$ and $(\tau_i)_{i=1,2}$ respectively. The following assertions are true.
  \begin{enumerate}
  \item \label{one} If $B \subseteq A$ is a unital $\ast$-subalgebra and $\tau
    \in \CEP(A)$, then $\tau|_B \in \CEP(B)$.
  \item \label{two} If $\pi:A_1 \to A_2$ is a unital $\ast$-homomorphism such
    that $\tau_2 \circ \pi = \tau_1$ and $\tau_1 \in \CEP(A_1)$, then
    $\tau_2|_{\pi(A_1)} \in \CEP(\pi(A_1))$.
  \item \label{three} If $\tau_1 \in \CEP(A_1)$ and $\tau_2 \in \CEP(A_2)$, then
    $\tau_1 \otimes \tau_2 \in \CEP(A_1 \odot A_2)$ and $\tau_1*\tau_2 \in
    \CEP(A_1*A_2)$ where $*$ denotes the reduced free product of tracial unital
    $\ast$-algebras \cite{NiSp}.
  \item \label{four} If $(\tau_n)_{n \in \N} \subset \CEP(A)$ is a sequence such
    that the pointwise limit $\tau := \lim_{n\to \infty} \tau_n$ exists, then
    $\tau \in \CEP(A)$.
  \end{enumerate}
\end{lemma}

\begin{proof}
  \eqref{one} and \eqref{two} follow from the fact that the Connes embedding
  property is stable under (trace-preserving) inclusions of von Neumann
  algebras.  \eqref{three} follows from the fact that the Connes embedding
  property is stable under tensor products and free products of von Neumann
  algebras with respect to tracial states \cite{Popa,v3}.  \eqref{four} is a
  direct ultra product construction.  Alternately, this readily follows from the
  definition of matricial microstates.
\end{proof}

\subsection{Hyperlinear discrete quantum groups}

Now let $\G$ be a compact quantum group (of Kac type) and consider the unital
Hopf $\ast$-algebra $A = \Pol(\G)$, with coproduct $\Delta$ (from now on we drop the notation $\Delta_0$ and simply write $\Delta$).  Given a tracial
state $\tau:\Pol(\G) \to \C$, we will write $\tau \in \CEP(\G)$ if $\tau \in
\CEP(\Pol(\G))$.  Our main interest is in determining when the Haar state $h_\G$ belongs to $\CEP(\G)$. I.e., when $(L^\infty(\G), h_\G)$ is a Connes embeddable von
Neumann algebra.  In this case, (following the analogies with discrete groups
given in Remark \ref{rem:discrete}) we will say that the dual discrete quantum
group $\hat \G$ is {\it hyperlinear}.

We start with a crucial but elementary lemma.  Given two states $\tau_1,\tau_2$
on $\Pol(\G)$, recall that their \textit{convolution product}
$\tau_1\star\tau_2: = (\tau_1 \otimes \tau_2) \circ \Delta$ is again a state on
$\Pol(\G)$ and $\tau_1\star \tau_2$ is tracial if both $\tau_1,\tau_2$ are.

\begin{lemma} \label{lemma:convolution} If $\tau_1,\tau_2 \in \CEP(\G)$, then
  $\tau_1 \star \tau_2 \in \CEP(\G)$.
\end{lemma}

\begin{proof} Let $\sigma = \tau_1 \otimes \tau_2|_{\Delta(\Pol(\G))}$.  Then it
  follows from Lemma \ref{lemma:stabilityproperties} \eqref{one}-\eqref{two}
  that $\sigma \in \CEP(\Delta(\Pol(\G)))$.  Moreover, since
  $\Delta^{-1}:\Delta(\Pol(\G)) \to \Pol(\G)$ is a $\ast$-isomorphism such that
  $(\tau_1\star\tau_2)\circ \Delta^{-1} = \sigma$, another application of Lemma
  \ref{lemma:stabilityproperties} \eqref{two} gives $\tau_1\star\tau_2 \in
  \CEP(\G)$.
\end{proof}

As one might expect, duals of coamenable compact quantum groups of Kac type are
always hyperlinear.

\begin{lemma}\label{lemma:coamenable}
  Let $\G$ be a coamenable compact quantum group of Kac type.  Then $\hat \G$ is
  hyperlinear.
\end{lemma}

\begin{proof}
  If $\G$ is of Kac type and coamenable, it follows from Ruan \cite[Proposition
  2.3]{ruan} that $(L^\infty(\G), h_\G)$ is a hyperfinite tracial von Neumann
  algebra.  In particular, this implies that there is a Haar state-preserving
  embedding of $L^\infty(\G)$ into the hyperfinite II$_1$-factor $R$.  Since $R$
  embeds trivially in $R^\omega$ in a trace-preserving way, we are done.
\end{proof}

\subsection{Quantum subgroups and a stability result for hyperlinearity}

In this section we present a new stability result for hyperlinear discrete
quantum groups (Theorem \ref{corollary:CEPsgs}).  The main conceptual tool here is a
quantization of the notion of a compact group being topologically
  generated by a pair of closed subgroups.  The results of this section will be
applied to specific examples in the next section.

Let $\G$ be a compact quantum group and $\G_1 \le \G$ a quantum subgroup (given
by a surjective $\ast$-homomorphism of Woronowicz C$^\ast$-algebras $\pi:C^u(\G)
\to C^u(\G_1)$).  Recall that any representation $u$ of $\G$ induces a
representation $u^{\G_1} := (\iota \otimes \pi)u$ of the quantum subgroup
$\G_1$.  When considering spaces of intertwiners, note that we always have the
inclusions
\begin{align} \label{sg:inclusion}\Hom_\G(u,v) \subseteq \Hom_{\G_1}(u^{\G_1},
  v^{\G_1})
\end{align}
for any pair of representations $u,v$ of $\G$.  If, moreover, we have equality
in \eqref{sg:inclusion} for every $u,v$, then it follows that $\G$ and $\G_1$
are isomorphic compact quantum groups.  This leads us to the following quantum
analogue of a compact group being topologically generated by a pair of closed subgroups.

\begin{definition} \label{qsg} Let $\G$ be a compact quantum group and
  $\G_1,\G_2 \le \G$ a pair of quantum subgroups.  We say that $\G$ is
  \textit{topologically generated by $\G_1$ and $\G_2$} (and write $\G = \langle
  \G_1,\G_2\rangle$) if
  \[\Hom_\G(u,v) = \Hom_{\G_1}(u^{\G_1}, v^{\G_1}) \cap \Hom_{\G_2}(u^{\G_2},
  v^{\G_2})\] for every pair of finite dimensional unitary representations $u,v$
  of $\G$.
\end{definition}

The following proposition shows that the condition $\G = \langle
\G_1,\G_2\rangle$ can be characterized purely in terms of a relation between the
Haar states on $\G_1, \G_2,$ and $\G$. If $v$ is a representation of $\G$ we
denote $\Fix(v) = \Hom_\G(1,v)$ the space of fixed vectors of $v$.

\begin{proposition} \label{proposition:Haar} Let $C$ be a class of
  representations of $\G$ such that any irreducible representation of $\G$ is
  equivalent to a subrepresentation of some element of $C$. Let $\G_1,\G_2 \le
  \G$. Denote $h = h_\G$ the Haar state of $\G$ and $h_i = h_{\G_i} \circ \pi_i$
  the state on $C^u(\G)$ induced by the Haar state of $\G_i$. Then the following
  conditions are equivalent.
  \begin{enumerate}
  \item \label{sg} $\G = \langle \G_1,\G_2 \rangle$.
  \item \label{sgFix} $\Fix(v^{\G_1}) \cap \Fix(v^{\G_2}) = \Fix(v)$ for all $v
    \in C$.
  \item \label{sgHaar} On $\Pol(\G)$, we have $h = \lim_{k \to \infty} (h_1\star
    h_2)^{\star k}$ (pointwise).
  \end{enumerate}
\end{proposition}

\begin{proof}
  \eqref{sg}$\implies$ \eqref{sgFix} is clear. 
  
 \eqref{sgFix}$\implies$
  \eqref{sg}. We have $\Fix(v\oplus w) = \Fix(v) \oplus \Fix(w)$ and, if $w$ is
  the restriction of $v$ on a subspace $K \subset H$, $\Fix(w) = \Fix(v) \cap
  K$. Hence the property $\Fix(v^{\G_1}) \cap \Fix(v^{\G_2}) = \Fix(v)$ holds
  for any finite dimensional subrepresentation of $v$. \eqref{sg} now follows by Frobenius
  reciprocity (see \cite[Proposition 3.4]{wo2}): we have indeed $\Hom_\G(u,v)
  \simeq \Fix(v\otimes \bar u)$ and similarly for the restrictions to $\G_1$,
  $\G_2$.

  \eqref{sgFix}$\implies$ \eqref{sgHaar}.  Let $u \in \mc B(H) \otimes
  C^u(\G)$ be a finite dimensional unitary representation of $\G$ and consider
  the operators $P = (\iota \otimes h)(u)$ and $P_i = (\iota \otimes h_i)(u)$
  $(i=1,2)$ in $\mc B(H)$.  Then $P, P_1, P_2$ are orthogonal projections with
  range equal to $\Fix u, \Fix u^{\G_1}, \Fix u^{\G_2}$, respectively.  In
  particular, $\lim_{k \to \infty}(P_1P_2)^k$ exists in $\mc B(H)$ and is the
  orthogonal projection with range equal to $\Fix (u^{\G_1}) \cap \Fix
  (u^{\G_2})$. As a result $P = \lim_{k \to \infty}(P_1P_2)^k$.  But since
  \begin{equation} \label{reduction} (P_1P_2)^k = (\iota \otimes (h_1\star
    h_2)^{\star k})(u),
  \end{equation}
  we conclude that $ h = \lim_{k \to \infty} (h_1\star h_2)^{\star k} $ on every
  matrix element of every finite dimensional unitary representation of $\G$.
  This proves the assertion.
  
    \eqref{sgHaar}$\implies$ \eqref{sgFix}. Similarly,
  \eqref{reduction} shows that $P = \lim_{k \to \infty}(P_1P_2)^k$ hence $\Fix
  (u^{\G_1}) \cap \Fix (u^{\G_2}) = Fix (u)$.
\end{proof}

\begin{remark}\label{remark:inner-faithfulness}
  At this point it is worthwhile pointing out the connection between our notion
  of topological generation by subgroups and the concept of an inner faithful
  representation of a Woronowicz C$^\ast$-algebra. Let $\G$ be a compact quantum
  group and $B$ a unital C$^\ast$-algebra. Recall that a $*$-homomorphism
  $\alpha : C^u(\G) \to B$ is {\em inner faithful} if $\Ker\alpha$ does not contain
  any non-zero Hopf $*$-ideal. Equivalently, for any factorization $\alpha =
  \tilde \alpha \circ \pi$ with $\pi : C^u(\G) \to C^u(\H)$, a surjective
  morphism of Woronowicz $C^*$-algebras, we have in fact that $\pi$ is an
  isomorphism. More generally, the {\em Hopf image} of $\alpha$ is the
  ``biggest'' quantum subgroup $(\H,\pi)$ of $\G$ such that $\alpha$ factors
  through $\pi : C^u(\G) \to C^u(\H)$, cf. \cite{BaBi}, and $\alpha$ is inner
  faithful iff its Hopf image is $(\G,\iota)$.

  With this terminology, it follows that $\G$ is topologically generated by
  $(\H_1,\pi_1)$, $(\H_2,\pi_2)$ iff $\alpha := (\pi_1\otimes\pi_2) \circ \Delta
  : C^u(\G) \to C^u(\H_1)\otimes C^u(\H_2)$ is inner faithful. Indeed, by
  \cite[Theorem~8.6]{BaBi} $\alpha$ is inner faitful iff $\Fix_\G(v) =
  \Fix(v^\alpha)$ for all representations $v$ of $\G$, where $v^\alpha =
  v^{\H_1}_{12}v^{\H_2}_{13}$. Then we have
  \begin{align*}
    \xi \in \Fix(v^\alpha) ~~~ &\Leftrightarrow ~~~
    v^{\H_1*}_{12} (\xi\otimes 1\otimes 1) = 
    v^{\H_2}_{13}(\xi\otimes 1\otimes 1)  \\
    & \Leftrightarrow ~~~ v^{\H_1} (\xi\otimes 1) = \xi\otimes 1=
    v^{\H_2}(\xi\otimes 1)
  \end{align*}
  so that $\Fix(v^\alpha) = \Fix(v^{\H_1}) \cap \Fix(v^{\H_2})$.

  More generally, we say that a quantum subgroup $(\H,\pi)$ is topologically
  generated by $(\H_1,\pi_1)$ and $(\H_2,\pi_2)$ if it is the Hopf image of
  $\alpha = (\pi_1\otimes\pi_2) \circ\Delta$. If $(\H,\pi)$, $(\H_1,\pi_1)$ are
  two quantum subgroups of $\G$, we say that $\H$ contains $\H_1$ if $\pi_1$
  factors through $\pi$. From the definitions, we have that if $\H = (\H,\pi)$ contains
  $\H_1$ and $\H_2$ then it also contains the subgroup generated by $\H_1$ and
  $\H_2$. Indeed, writing $\pi_1 = \rho_1\circ\pi$ and $\pi_2 = \rho_2\circ\pi$
  we have $\alpha = (\rho_1\otimes\rho_2) \circ (\pi\otimes\pi) \circ \Delta_\G
  = (\rho_1\otimes\rho_2)\circ \Delta_\H \circ \pi$ and $\pi$ is a morphism of
  Woronowicz $C^*$-algebras, hence $\H$ contains the Hopf image of $\alpha$.
\end{remark}

We now turn to an interesting corollary of Proposition \ref{proposition:Haar},
which shows that it is possible to deduce the Connes embeddability of
$L^\infty(\G)$ from the Connes embeddability of the von Neumann algebras
associated to its quantum subgroups.

\begin{theorem} \label{corollary:CEPsgs} Let $\G$ be a compact quantum group
  of Kac type and assume $\G = \langle \G_1,\G_2 \rangle $ for some pair of
  quantum subgroups $\G_1,\G_2 \le \G$.  If $\hat \G_1$ and $\hat \G_2$ are
  hyperlinear, then so is $\hat \G$.
\end{theorem}

\begin{proof}
  Consider the Haar states $h_{\G_1}$ and $h_{\G_2}$, and the associated states
  $h_1$, $h_2$ on $\Pol(\G)$.  Since the GNS construction for $h_i$ yields
  $(L^\infty(\G_i),h_{\G_i})$, which is Connes embeddable by assumption, we
  conclude that $h_1,h_2 \in \CEP(\G)$.  By Lemma \ref{lemma:convolution} $(h_1
  \star h_2)^{\star k} \in \CEP(\G)$ for all $k \in \N$.  Since $\G = \langle
  \G_1,\G_2 \rangle $, an application of Proposition \ref{proposition:Haar} and
  \ref{lemma:stabilityproperties}\eqref{four} shows that $h_{\G} = \lim_{k \to
    \infty} (h_1\star h_2)^{\star k}$ belongs to $\CEP(\G)$.
\end{proof}

\begin{remark} \label{rem:explicit}
An examination of Proposition \ref{proposition:Haar}, Theorem \ref{corollary:CEPsgs} and their proofs shows that under the above assumptions, one can build matricial microstates for generators of $L^\infty(\G)$ using matricial microstates for elements of  tensor products of $L^\infty(\G_1)$ and $L^\infty(\G_2)$ (which, by assumption, are known to exist!).  To see this, note that by a standard ultraproduct argument along the lines of \cite[Prop. 3.3]{MR2465797}, it suffices to exhibit  a Haar-state-preserving embedding of $L^\infty(\G)$ into a tracial ultraproduct of tensor products of the von Neumann algebras $(L^\infty(\G_1),h_{\G_1})$ and $(L^\infty(\G_2),h_{\G_2})$.

  To this end, for each $k \in\N$, let $M_k$ be the finite von Neumann algebra $(L^\infty(\G_1) \overline{\otimes} L^\infty(\G_2))^{\overline{\otimes}k}$ equipped with the faithful normal trace-state $\tau_k := (h_{\G_1} \otimes h_{\G_2})^{\otimes k}$.  Let $\omega$ be a fixed free ultrafilter on $\N$ and consider the tracial ultraproduct \[M = \prod_{k \in \N}^\omega(M_k, \tau_k) := \Big(\prod_{k \in \N}^{\ell^\infty}(M_k, \tau_k)\Big) /I_\omega, \] where $I_\omega = \{(x_k)_{k \in \N} \in \prod_{k \in \N}^{\ell^\infty}(M_k, \tau_k) \ : \ \lim_{k \to \omega}\tau_k(x_k^*x_k) = 0  \}$.  Note that $M$ is a finite von Neumann algebra with  faithful normal tracial state $\tau_{\omega}(x) = \lim_{k \to \omega} \tau_k(x_k)$, where $x = (x_k)_{k \in \N}^{^\centerdot} \in M$ denotes the equivalence class of $(x_k)_{k \in \N} \in  \prod_{k \in \N}^{\ell^\infty}(M_k, \tau_k)$.  

Denote by $\sigma_k:\Pol(\G) \to M_k$ the unital $\ast$-homomorphism \[\sigma_k(x) =  \big((\pi_{h_{\G_1}} \circ \pi_1) \otimes (\pi_{h_{\G_2}} \circ \pi_2)\big)^{\otimes k} \circ \Delta^{2k-1}(x) \qquad (x \in \Pol(\G)),  \] where  $\Delta^r := (\iota \otimes \Delta) \circ \Delta^{r-1}:\Pol(\G) \to Pol(\G)^{\otimes (r+1)}$ is the $r$-fold iterated coproduct,  $\pi_i:\Pol(\G) \to \Pol(\G_i)$ is the surjective $\ast$-homomorphism identifying $\G_i$ as a quantum subgroup of $\G$, and $\pi_{h_{\G_i}}:\Pol(\G_i) \to L^\infty(\G_i)$ is the GNS representation associated to $h_{\G_i}$.  
Since  $\G = \langle \G_1, \G_2 \rangle$ by assumption, we have $h_\G = \lim_{k \to \infty}\tau_k \circ \sigma_k$ and therefore the $\ast$-homomorphism
\[\sigma:(\Pol(\G), h_\G) \to (M,\tau_\omega); \qquad \sigma(x) = (\sigma_k(x))_{k \in \N}^{^\centerdot} \qquad (x \in \Pol(\G))\] is trace-preserving and  extends uniquely to a trace-preserving normal injective $\ast$-homomorphism $\sigma:(L^\infty(\G), h_\G) \to (M,\tau_\omega)$. \end{remark} 

\section{The Connes embedding property for $L^\infty(O_N^+)$ and $L^\infty(U_N^+)$}
\label{section:mainresult}

In this section we apply the general theory of the previous sections to study
the hyperlinearity of the discrete quantum groups dual to $O_N^+$ and $U_N^+$,
$N \ge 2$.  In particular, we prove that $O_N^+$ is topologically generated by
certain canonical pairs of quantum subgroups of lower rank.  The results of this
section may be of independent interest, particularly with respect to the problem
of classifying \textit{all quantum subgroups} $O_N \le \G \le O_N^+$ (see Section \ref{section:apps} for more on this).  

Below we
will consider the following list of quantum subgroups of $O_N^+$. Recall that we
denote $u \in B(H)\otimes C^u(O_N^+)$ the fundamental representation of $O_N^+$,
with $H = \C^N$, and let us also put $S_1 = \{\xi \in \R^N \mid \|\xi\|=1\}
\subset H$.

\begin{enumerate}
\item The classical orthogonal group $O_N \le O_N^+$, given by the Woronowicz
  C$^*$-morphism $\pi_{O_N} : C^u(O_N^+) \to C(O_N)$ whose kernel is generated by
  commutators.
\item The classical permutation group $\mathfrak{S}_N \le O_N^+$, given by the
  Woronowicz C$^*$-morphism $\pi_{\mathfrak{S}_N} : C^u(O_N^+) \to
  C(\mathfrak{S}_N)$ whose kernel is generated by the commutators together with
  the elements $(u_{ij}-u_{ij}^2)_{1 \le i, j \le N}$.
\item The free product quantum subgroups $O_a^+ \hatfree O_b^+\le O_N^+$ for
  $a+b=N$, given by the Woronowicz C$^*$-morphism $\pi_{a,b}:C^u(O_N^+) \to
  C^u(O_a^+ \hatfree O_b^+)$ which sends the $a \times a$ upper left (resp. $b
  \times b$ lower right) corner of the fundamental representation of $O_N^+$ to
  the fundamental representation of $O^+_a$ (resp. $O^+_b$), and all other
  entries to $0$.
\item The quantum stabilizer subgroups $O_{N-1}^{+,\xi}\le O_N^+$ for $\xi\in
  S_1$, given by the Woronowicz C$^*$-morphisms $\pi_\xi : C^u(O_N^+) \to
  C^u(O_{N-1}^{+,\xi})$ obtained by completing $\xi$ into an orthonormal basis and
  sending the corresponding generator $u_{11}$ to $1$.  Note that
  $O_{N-1}^{+,\xi} \simeq O_{N-1}^+$ for all $\xi$.
\end{enumerate}

The main theorems of this section are as follows.

\begin{theorem} \label{theorem:N>3} Let $N \ge 4$, then the following assertions
  are true.
  \begin{enumerate}
  \item $O_N^+ = \langle O_N, O_{N-1}^{+,\xi} \rangle$ for each $\xi \in S_1$.
  \item For any pair of linearly independent vectors $\xi_1, \xi_2 \in S_1$, $O_N^+ = \langle O_{N-1}^{+,\xi_1},
    O_{N-1}^{+,\xi_2} \rangle$.
  \end{enumerate}
\end{theorem}

\begin{theorem} \label{theorem:N=4} 
   Let $n \ge 2$ be a non-negative integer.  Then $O_{2n}^+ = \langle
   \mathfrak{S}_{2n}, O_n^+ \hatfree O_n^+ \rangle$.
\end{theorem}

Before proving Theorems \ref{theorem:N>3} and
\ref{theorem:N=4} we state their applications to hyperlinearity.

\begin{corollary} \label{CEPorthogonal} Let $N = 2$ or $N \ge 4$.  Then
  $\widehat{O_N^+}$ is hyperlinear.
\end{corollary}

\begin{proof}
  The hyperlinearity of $\widehat{O_2^+}$ follows from Lemma
  \ref{lemma:coamenable}.  For the case $N=4$, note that $O_4^+ =
  \langle\mathfrak S_4,O_2^+\hat{*}O_2^+ \rangle$ by Theorem
  \ref{theorem:N=4}. 
  Since $L^\infty(\mathfrak S_4)$ and $L^\infty(O_2^+\hat{*}O_2^+) =
  (L^\infty(O_2^+),h_{O_2^+})*(L^\infty(O_2^+),h_{O_2^+})$ are both Connes
  embeddable, we conclude that $\widehat{O_4^+}$ is hyperlinear by Theorem
  \ref{corollary:CEPsgs}.  Finally, the cases $N \ge 5$ follow by induction
  using Theorem \ref{theorem:N>3} and Theorem \ref{corollary:CEPsgs}.
\end{proof}

Using a structure result of Banica \cite[Th\'eor\`eme 1]{ba2}, we can easily deduce the hyperlinearity of $\widehat{U_N^+}$ from the
corresponding result for $\widehat{O_N^+}$.

\begin{theorem} \label{CEPunitary} Let $N=2$ or $N \ge 4$.  Then
  $\widehat{U_N^+}$ is hyperlinear.
\end{theorem}

\begin{proof}
  From \cite[Th\'eor\`eme 1]{ba2}, there exists trace-preserving embedding
  $(L^\infty(U_N^+), h_{U_N^+}) \hookrightarrow (L^\infty(\mathbb T)\ast
  L^\infty(O_N^+), \tau \ast h_{O_N^+})$, where $\tau$ denotes integration with
  respect to the Haar probability measure on $\mathbb T$.  The Connes
  embeddability of $L^\infty(U_N^+)$ now follows from the Connes embeddability
  of $L^\infty(\mathbb T), L^\infty(O_N^+)$ and Lemma
  \ref{lemma:stabilityproperties} \eqref{three}.
\end{proof}

\begin{remark}
  We expect that Theorem \ref{theorem:N>3} holds when $N = 3$ (and therefore
  that $\widehat{O^+_3}, \widehat{U^+_3}$ are hyperlinear).  However, the
  following proof method seems to break down in this case. See also
  Remark~\ref{rem:roland}.
\end{remark}

The remainder of this section is devoted to proving the above quantum subgroup
generation results for $O_N^+$.

\subsection{Proofs of Theorems~\ref{theorem:N>3} and \ref{theorem:N=4}}
\label{sec_diff}

We begin by developing some tools for the proof of Theorem \ref{theorem:N>3}.

Recall that we denote $\Fix_k = \Hom_{O_N^+}(1,u^{\otimes k})$ where $u$ is the
fundamental representation of $O_N^+$, and let us denote similarly
\begin{align*}
  \Fix_k^\xi &= \Hom_{O_N^{+,\xi}}(1,u^{\otimes k}) =
  \Hom_{O_{N-1}^+}(1,(\iota\otimes\pi_\xi)(u)^{\otimes k}) \subseteq H^{\otimes k}, \\
  \Fix_k^{O_N} &= \Hom_{O_N}(1,(\iota\otimes\pi_{O_N})(u)^{\otimes k}) \subset
  H^{\otimes k}.
\end{align*}
According to Proposition~\ref{proposition:Haar}, Theorem~\ref{theorem:N>3} is
equivalent to the equalities $\Fix_k^\xi \cap \Fix_k^{O_N}$ $=$ $\Fix_k^{\xi_1}
\cap \Fix_k^{\xi_2}$ $=$ $\Fix_k$ for all $k$. Hence we start by describing the
subspaces $\Fix_k^\xi$.

Let $NC_{2,1}(k)$ be the set of non-crossing partitions of $\{1, \ldots, k\}$
consisting of blocks with cardinality at most $2$.  In what follows, a block of $p \in NC_{2,1}(k)$ with cardinality equal to $1$ will be called a {\it singleton}, and a block with cardinality equal to $2$ will be called a {\it pair}.  We also denote by $NC_{2,1}^s(k) \subset NC_{2,1}(k)$ the
subset of non-crossing partitions containing exactly $s$ singletons, so that $NC_2(k) =
NC_{2,1}^0(k)$. For $p\in NC_{2,1}(k)$ and $i$ a $k$-tuple we put $\delta^p_i =
1$ if $i_l = i_m$ for all pairs $\{l,m\} \in p$, and $\delta^p_i = 0$ else. Then
we associate to $p \in NC_{2,1}^s(k)$ a linear map $T_p : H^{\otimes s} \to
H^{\otimes k}$ as follows:
\begin{displaymath}
  T_p(\xi_1\otimes\cdots\otimes\xi_s) = \sum_{i_j = 1}^N \delta^p_i 
  (e_{i_1}\otimes \cdots\otimes \xi_1\otimes \cdots \otimes \xi_s\otimes 
  \cdots\otimes e_{i_k}),
\end{displaymath}
where we put a term $\xi_i$ at position $l$ if $\{l\}$ is the $i^{\text{th}}$
singleton in $p$, and a term $e_{i_l}$ else. In other words, $T_p$ is the usual
map associated to the pair partition (possibly with crossings) $p'\in P_2(s,k)$ obtained
from $p$ by attaching a vertical segment to each singleton. We will also denote
$T_p = T_p^1$ and consider the variant $T_p^2$ (resp. $T_p^3$) where the indices
$i_j$ range from $2$ to $N$ (resp. $3$ to $N$). Finally we denote $S :
H^{\otimes l} \to H^{\otimes l}$ the ``symmetrizing'' operators
\begin{displaymath}
  S : \xi_1\otimes\cdots\otimes\xi_l \mapsto \sum_{\sigma\in
    \mathfrak{S}_l} \xi_{\sigma(1)}\otimes\cdots\otimes\xi_{\sigma(l)}.
\end{displaymath}

\begin{lemma} \label{lem_fix_xi} Denote $v = (\iota\otimes\pi_{O_N})(u)$ the
  fundamental representation of $O_N$ and fix $\xi\in S_1$.
  \begin{enumerate}
  \item We have, for any $N \geq 2$, $k\in\N$:
    \begin{displaymath}
      \Fix_k^\xi =  
      \Span~ \{T_p(\xi^{\otimes s}) \mid p \in  NC_{2,1}^s(k), 0\le s\le k\}.
    \end{displaymath}
  \item The vectors $T_p(\xi^{\otimes s})$ for $p \in NC_{2,1}(k)$ are linearly
    independent if $N\geq 3$.
  \item We have $T_p \in \Hom_{O_N}(v^{\otimes s}, v^{\otimes k})$ for all $p
    \in NC_{2,1}^s(k)$.
  \end{enumerate}
\end{lemma}

\begin{proof}
  Consider $e_1 = \xi$ as the first vector of an ONB $(e_1,\ldots,e_n)$.  By
  definition we have $(\iota\otimes\pi_\xi)(u) = 1 \oplus w$ in the decomposition
  $H = \C e_1 \oplus e_1^\bot$, where $w$ is equivalent to the fundamental
  representation of $O_{N-1}^+$. As a result $(\iota\otimes\pi_\xi)(u)^{\otimes
    k}$ decomposes into pairwise orthogonal subrepresentations equivalent to
  $w^{\otimes k-s}$, $0\leq s\leq k$. We know that the subspace of fixed vectors
  of $w^{\otimes k-s}$ is spanned by the elements $T_q^2(1)$, $q \in
  NC_2(k-s)$. Now identifying $w^{\otimes k-s}$ with a subrepresentation of
  $(\iota\otimes\pi_\xi)(u)^{\otimes k}$ corresponds to inserting vectors $\xi$ at
  $s$ fixed legs of the tensor product $H^{\otimes k}$, and this maps $T_q^2(1)$
  to $T_r^2(\xi^{\otimes s})$, where $r$ is obtained from $q$ by inserting $s$
  singletons at fixed places. In this way we obtain all partitions of
  $NC_{2,1}^s(k)$.

  We know by Theorem~\ref{thm_invariants} that the vectors $T_q^2(1)$, $q \in
  NC_2(k-s)$, are linearly independent for $N-1\geq 2$. Since we have decomposed
  $(\iota\otimes\pi_\xi)(u)^{\otimes k}$ into orthogonal subrepresentations,
  this implies that the family of vectors $T_r^2(\xi^{\otimes s})$, $r \in
  NC_{2,1}(k)$, is linearly independent. Now we observe that the vectors
  $T_p(\xi^{\otimes s}) = T_p^1(\xi^{\otimes s})$, $p \in NC_{2,1}(k)$, can be
  decomposed as linear combinations of the vectors $T_r^2(\xi^{\otimes s})$ by
  writing
  \begin{displaymath}
    \sum_{i=1}^N e_i\otimes e_i = \xi\otimes\xi + \sum_{i=2}^N e_i\otimes e_i
  \end{displaymath}
  at each pair of legs of $H^{\otimes k}$ determined by the pairs in $p$. Note
  that the partitions $r\neq p$ used to decompose in this way a vector
  $T_p(\xi^{\otimes s})$ have strictly more singletons than $p$, so that the
  decomposition matrix is block triangular (with respect to the value of $s$)
  with identity blocks on the diagonal. This implies that the family
  $T_p(\xi^{\otimes s})$, $p \in NC_{2,1}(k)$, is linearly independant, and
  spans the same subspace as the vectors $T_r^2(\xi^{\otimes s})$. This proves
  the first two assertions.

  Finally, for any $p \in NC_{2,1}^s(k)$ we know that $T_p = T_{p'}$ for a
  suitable partition $p' \in P_2(s,k)$, see above, and that the maps $T_{p'}$
  are $O_N$-intertwiners, see the end of Section~\ref{sec_inv_theory}.
\end{proof}

We now reduce Theorem~\ref{theorem:N>3} to a linear independence problem:

\begin{proposition}\label{diff}
  For $N\geq 3$ and $k \in \N$, the following are equivalent:
  \begin{enumerate}
  \item We have $\Fix_k^{\xi_1} \cap \Fix_k^{\xi_2} = \Fix_k$ for some (or any)
    pair of linearly independent vectors $\xi_1, \xi_2 \in S_1$;
  \item We have $\Fix_k^\xi \cap \Fix_k^{O_N} = \Fix_k$ for some (or any)
    $\xi\in S_1$;
  \item The vectors $T_p(S(e_1\otimes\cdots\otimes e_1\otimes e_2))$, $p \in
    NC_{2,1}(k) \setminus NC_2(k)$, are linearly independent.
  \end{enumerate}
\end{proposition}

\begin{proof} 
  We first recall that two different stabilizer subgroups $O_{N-1} < O_N$
  generate $O_N$. Indeed, for $1\leq i<j\leq N$, calling $R_{i,j,\theta}$ the
  rotation of angle $\theta$, between the canonical basis vectors $e_i$ and
  $e_j$ it is known that $R_{i,j,\theta}$ generate $SO_N$ if one takes all
  $1\leq i<j\leq N,\theta\in [0,2\pi )$.

  Without loss of generality -- at the possible cost of involving conjugation by
  rotations -- we can assume that the first copy of $O_{N-1}$ fixes $e_N$ and
  the second copy fixes $e_1$.  One can check that $R_{1,N,\theta}$ can be
  obtained as a conjugation of $R_{1,N-1,\theta}$ by $R_{N-1,N,\pi}$.  This
  implies that any two copies of $SO_{N-1}< SO_N$ generate $SO_N$.  The
  fact that two different copies of $O_{N-1} < O_N$ generate $O_N$ follows
  from the fact that we can find in $O_{N-1}$ an isometry that takes $SO_{N-1}$
  to $O_{N-1}\backslash SO_{N-1}$ by left multiplication.
  
  Now let $x \in \Fix_k^{\xi_1} \cap \Fix_k^{\xi_2}$ be given.  Then $x$ is
  fixed by the two copies of $O_{N-1}$ inside the quantum subgroups
  $O_{N-1}^{+,\xi_1}$ and $O_{N-1}^{+,\xi_2}$, hence it is fixed by $O_N$ by the
  previous paragraph. On the other hand, for $g \in O_N$ we have $\Fix_k^{g\xi}
  = g \cdot \Fix_k^\xi := v(g)^{\otimes k} \Fix_k^\xi$, where $v = u^{O_N}$ is
  the fundamental representation of $O_N$. As a result, if $x \in \Fix_k^{O_N}
  \cap \Fix_k^\xi$, then $x$ lies in $\Fix_k^{\zeta}$ for any $\zeta \in
  S_1$. This shows that $\Fix_k^{\xi_1} \cap \Fix_k^{\xi_2}$ and $\Fix_k^{O_N}
  \cap \Fix_k^\xi$ are equal and independent of the choice of $\xi$ and the
  linearly independent pair $\xi_1, \xi_2$ in $S_1$.

  According to the previous lemma, any $x \in \Fix_k^{\xi}$ can be written $x =
  \sum \lambda_p T_p(\xi^{\otimes s})$ in a unique way, and since $\Fix_k$ is
  spanned by the vectors $T_p(1)$, $p \in NC_2(k)$, we have $x \in \Fix_k$ \iff
  $\lambda_p = 0$ for all $p \in NC_{2,1}^s(k)$, $s>0$. Besides, if $x$ is
  $O_N$-invariant we have $x = g \cdot x = \sum \lambda_p T_p((g\xi)^{\otimes
    s})$ for all $g \in O_N$, so that the map $\Tt_\lambda : \R^N \to
  (\R^N)^{\tens k}$, $\xi \mapsto \sum \lambda_p T_p(\xi^{\otimes s})$ is
  constant on $S_1$. Hence the second assertion in the statement is equivalent
  to the implication (I) ``$\Tt_\lambda$ constant on $S_1$ $\Rightarrow$
  $\lambda_p = 0$ for $p\in NC_{2,1}^s(k)$, $s>0$'' for all $\lambda :
  NC_{2,1}(k) \to \C$.

  Now we differentiate: $\Tt_\lambda$ is constant on $S_1$ \iff $d_\xi
  \Tt_\lambda(\eta) = 0$ for all $\xi\in S_1$, $\eta\bot\xi$. Moreover by
  $O_N$-covariance of $\Tt_\lambda$ we have $g\cdot d_\xi \Tt_\lambda(\eta) =
  d_{g\xi}\Tt_\lambda(g\eta)$, and since $O_N$ acts transitively on pairs of
  normed orthogonal vectors, $\Tt_\lambda$ is constant on $S_1$ \iff $d_{e_1}
  \Tt_\lambda(e_2) = 0$. Then we compute $d_\xi(\xi^{\otimes s})(\eta) =
  S(\xi\otimes\cdots\otimes\xi\otimes\eta)/(s-1)!$, hence
  \begin{displaymath}
    d_\xi\Tt_\lambda(\eta) = \sum_{s>0,~ p\in NC_{2,1}^s(k)} \frac{\lambda_p}{(s-1)!}
    ~ T_p(S(\xi^{\otimes s-1}\otimes\eta)).
  \end{displaymath}
  This shows the equivalence of the last assertion in the statement with the
  condition (I) above.
\end{proof}

\begin{proof}[Proof of Theorem~\ref{theorem:N>3}]
  We will verify the linear independence condition given in Part (3) of Proposition \ref{diff}.  Consider the vectors $y_{p,i} = T_p^3(e_{i_1}\otimes\cdots\otimes e_{i_s})
  \in H^{\otimes k}$ with $i_l = 1, 2$ and $p \in NC_{2,1}(k)$. They form a
  linearly independent family, which we shall denote by $\Cc$. Indeed if $\langle y_{p,i}| y_{q,j}\rangle \neq 0$,
  then $p$, $q$ must have the same singletons and $i = j$. Moreover when this is
  the case, then $\langle y_{p,i}| y_{q,j}\rangle$ coincides with the scalar product
  $\langle T'_{p'}(1)|T'_{q'}(1)\rangle$ associated with the partitions $p'$, $q' \in
  NC_2(k-s)$ obtained from $p$ and $q$ by removing singletons, where $T'_{p'}$
  is the map analogous to $T_{p'}$, but in dimension $N-2$. Since $N-2\geq 2$, the
  vectors $T'_{p'}(1)$ are linearly independent, and we can deduce that the
  Gram matrix of the family $\Cc$ is invertible (cf. Theorem \ref{thm_invariants}).
  
  Now consider the vectors $x_p = T_p(S(e_1\otimes\cdots\otimes e_1\otimes e_2))/(s-1)!$ from Proposition~\ref{diff}.  Note that each $x_p$ can be written as a (unique) linear combination of elements in $\Cc$. This follows from the definition of $S$ and by writing $\sum_{i=1}^N e_i\otimes e_i = e_1\otimes e_1 + e_2\otimes
  e_2 + \sum_{i=3}^N e_i\otimes e_i$ as in the proof of Lemma~\ref{lem_fix_xi}. More
  precisely, if $p\in NC_{1,2}^s(k)$ then $x_p$ decomposes into the sum of the $s$ vectors $y_{p,i}$ with $i$   taking the value $2$ only once, and a linear combination of vectors $y_{q,j}$ with $q$ having
  strictly more singletons that $p$. As a result, if we partially order the families $\Bb =
  (x_p)$ and $\Cc = (y_{p,i})$ according to the number of singletons $s$ in
  $p$, the corresponding decomposition matrix for $\Bb$ in terms of the basis $\Cc$ will be block lower-triangular (with rectangular blocks),
  and each diagonal sub-block (one for each integer $s$) is itself block diagonal,
  with diagonal blocks which are non-zero columns (one for each partition $p \in
  NC_{2,1}^s(k)$). In particular, this decomposition matrix has maximal rank and therefore $\Bb$ is linearly independent.
\end{proof}

\begin{remark} \label{rem:roland}
  Although the proof above only applies for $N\geq 4$, it seems very likely that
  the linear independence condition introduced in Proposition~\ref{diff}, and
  hence Theorem~\ref{theorem:N>3}, also hold at $N=3$. This would imply the
  hyperlinearity of $\hat O_N^+$ and $\hat U_N^+$ for all $N\geq 2$, without
  relying on Theorem~\ref{theorem:N=4}. In fact, we have strong numerical
  evidence that the family of vectors $T_p(e_1)$ associated to ``one singleton''
  partitions $p \in NC_{2,1}^1(k)$ is linearly independant for all $N\geq 2$,
  and using similar techniques as above this would imply
  Theorem~\ref{theorem:N>3} for all $N\geq 3$.
\end{remark}

We now provide a proof of Theorem~\ref{theorem:N=4}.  

\begin{proof}[Proof of Theorem~\ref{theorem:N=4}]
  Let $u \in \mc B(\C^{2n}) \otimes C^u(O_{2n}^+)$ be the fundamental representation
  of $O_{2n}^+$, let $l\in \N$, and put $u_{l} = u^{\otimes l}$.  Similarly, let
  $w$ be the fundamental representation of $U_{2n}^+$ and put $w_{l} = w \otimes
  \bar w \otimes w \otimes \ldots$ ($l$-terms).  In what follows, we will regard $w_{l}$ and $u_{l}$ as
  both acting on the same Hilbert space (cf. Remark \ref{rem:unitaryinvariants}).  In
  particular, when $l$ is even, we have $\Fix(u_{l}) = \Fix(w_{l})$ under this identification. Moreover, from the description of the spaces of intertwiners for free products of compact
  quantum groups given in \cite[Proposition~2.15]{Lemeux}, 
  it also follows that $\Fix(u_{l}^{O_n^+\hatfree O_n^+}) = \Fix(w_{l}^{U_n^+\hatfree
    U_n^+})$ when $l$ is even.

  Now choose $x \in \Fix(u_{l}^{\mathfrak S_{2n}}) \cap
  \Fix(u_{l}^{O_n^+\hat{*}O_n^+}) \subset (\C^{2n})^{\otimes l}$.  Our goal,
  according to Proposition \ref{proposition:Haar} \eqref{sgFix}, is to show that
  $x \in \Fix(u_{l})$.  Since $\Fix(u_{l}^{O_n^+\hat{*}O_n^+}) =  \Fix(u_{l}) = \{0\}$ when $l$ is odd, we will assume $l=2k$ ($k \in \N$) for the remainder of the proof.    According to the discussion in the previous paragraph,
  we have $x \in \Fix(w_{2k}^{U_n^+\hatfree U_n^+}) \subset
  \Fix(w_{2k}^{U_n\times U_n})$, where the classical group $U_n \times U_n \le U_n^+ \hat * U_n^+$ acts block-diagonally on
  $\C^{2n}$ with respect to a fixed orthonormal basis $(e_i)_{i=1}^{2n}$.  Now we recall the following elementary group theoretic fact.

  \begin{align} \label{eqn:generate}\text{$U_{2n}$ is generated by the subgroups $\mathfrak S_{2n}$
    and $U_n \times U_n$.}
\end{align}

  One can actually even show more, namely that for any $d\geq 2$, $U_d$ is
  (algebraically) generated by the subgroups $\mathfrak S_d$ and $U_2$, where
  $U_2$ is viewed as sitting on the upper left corner of $d\times d$
  matrices. This is trivial for $d=2$. For general $d$, we proceed by induction
  over $d$: given an element of $U_d$, it is possible to multiply it on the left
  by $d-1$ elements of type $\sigma U \sigma^{-1}$ where $\sigma \in \mathfrak
  S_d$ and $U\in U_2$, and ensure that the bottom element of the last column of
  the new element of $U_d$ is $1$.  Indeed, the action by left multiplication
  leaves the columns invariant, and successive operations of the groups $U_2,
  (13)U_2(13), \ldots ,(1d)U_2 (1d)$ can be performed to ensure that the element
  of respective indices $(1,d), (2,d),\ldots , (d-1,d)$ is sent to zero, and in
  turn, that the entry of index $(d,d)$ is sent to $1$.  By orthogonality
  relations, the new matrix obtained has also zeros on all entries of the last
 row apart from the last one, therefore it sits in $U_{d-1}$ viewed as the
  upper left corner of $U_d$. The general result follows by induction.

  Applying fact \eqref{eqn:generate}, we finally conclude that $x \in
  \Fix(w_{2k}^{U_n^+\hatfree U_n^+}) \cap \Fix (w_{2k}^{U_{2n}})$.  To finish
  the proof, we appeal to the following lemma, which is a special case of a very
  recent result of Chirvasitu \cite{Ch}.  We include a detailed proof for the
  convenience of the reader.
\end{proof}

\begin{lemma}[\cite{Ch}, Lemma 3.11] \label{Ch}
  With the notation and conventions as above, we have the equalities 
  \[\Fix(w_{2k}^{U_n^+\hatfree U_n^+}) \cap \Fix
  (w_{2k}^{U_{2n}}) = \Fix(w_{2k}) =  \Fix(u_{2k}) \qquad (k \in \N).\]  
\end{lemma}

\begin{proof}
  Let $H = \C^{2n}$, fix an orthonormal basis $(e_i)_{i=1}^{2n}$ for $H$ and
  write $H = H_1 \oplus H_2$, where $H_1 = \text{span}(e_1, \ldots, e_{n})$ and
  $H_2 = \text{span}(e_{n+1},\dots, e_{2n})$.  Denote by $P_i \in \mc B(H)$ the
  orthogonal projection whose range is $H_i$. In the following, we will fix once
  and for all the linear isomorphism 
  \begin{equation} \label{lin_iso}
    \Phi:(H\otimes \bar H)^{\otimes k} \to \mc B(H^{\otimes k})
  \end{equation} 
  given by identifying an elementary tensor $e_{i_1} \otimes \overline{e_{i_2}}
  \otimes \ldots \otimes e_{i_{2k-1}} \otimes \overline{e_{i_{2k}}} \in (H
  \otimes \bar H)^{\otimes k}$ with the rank-one operator
  \begin{displaymath}
    H^{\otimes k} \owns \xi \mapsto \big(\otimes_{r=1}^k e_{i_{2r-1}}\big) 
    \langle \otimes_{r=1}^k e_{i_{2r}}\mid\xi\rangle \in H^{\otimes k}.
  \end{displaymath}
  Recall that $\Fix(w_{2k}^{U_{2n}})$ is spanned by the vectors $T_p = \sum_i
  \delta_i^p (e_{i_1}\otimes \overline{e_{i_2}} \otimes \ldots \otimes
  e_{i_{2k-1}} \otimes \overline{e_{i_{2k}}})$ where $p \in P_2(2k)$ is a pair
  partition respecting the additional requirement that even points are connected
  to odd points. 

  When $\Phi$ is restricted to the subspace $\Fix(w_{2k}^{U_{2n}})$, we obtain
  an isomorphism
  \begin{equation}\label{U_{2n}}\Fix(w_{2k}^{U_{2n}}) \cong^\Phi
    \Hom((w^{U_{2n}})^{\otimes k}, (w^{U_{2n}})^{\otimes k}),
  \end{equation} 
  which maps the vector $T_p$ to a map $T_q$, where $q \in P_2(k,k)$ is obtained
  by connecting the even (respectively, odd) points of $p$ to the input
  (respectively, output) points of $q$. We denote by $q \in S(k,k)$ the pair
  partitions obtained in this way: these are exactly the ones where input points
  are connected to output points via the permutation specified by $q$. Note that one recovers the Schur-Weyl duality
  for unitary groups describing $\Hom((w^{U_{2n}})^{\otimes k},
  (w^{U_{2n}})^{\otimes k})$ as the linear span of the operators $(T_q)_{q \in S(k,k)}$ which permute the tensor factors of $H^{\otimes k}$.

  On the other hand, the image by $\Phi$ of the subspace $\Fix(w_{2k})$ is
  spanned by the maps $T_q$ where $q$ belongs to a subset $S'(k,k) \subset
  S(k,k) \subset P_2(k,k)$: namely the one corresponding to $p \in NC_2(2k)
  \subset P_2(2k)$. The only thing we will need to know about $S'(k,k)$ is that
  the family $T_q$, $q \in S'(k,k)$, is linearly independent as soon as $n\geq
  2$. This is indeed the case since $(T_p)_{p \in NC_2(2k)}$ is linearly independent and $\Phi$ is an
  isomorphism.

  Finally, for each $q \in S'(k,k)$ and each $k$-tuple $i = (i_1, \ldots, i_{k})
  \in \{1,2\}^k$, we define a linear map $T_{q,i}:= T_q P_i$, where $P_i :=
  \otimes_{r=1}^k P_{i_r}$ is the orthogonal projection whose range is
  $H_i := \otimes_{r=1}^k H_{i_r}$.  From the description of the
  intertwiner spaces of $U_N^+$ ($N \ge 2$) and their free products given in
  \cite[Section 9]{BaCo} and \cite[Proposition 2.15]{Lemeux}, respectively, it
  follows that the family $T_{q,i}$, $q \in S'(k,k)$, $i \in \{1,2\}^k$, forms a
  basis of the intertwiner space 
  \begin{displaymath}
    \Hom((w^{U_n^+\hat *
      U_n^+})^{\otimes k}, (w^{U_n^+\hat * U_n^+})^{\otimes k}) \cong^\Phi 
    \Fix(w_{2k}^{U_n^+\hat * U_n^+}).
  \end{displaymath}

  In view of the above isomorphisms, it remains to demonstrate that any linear
  map $T \in $ $\Hom$ $((w^{U_n^+\hat * U_n^+})^{\otimes k}, (w^{U_n^+\hat *
    U_n^+})^{\otimes k}) \cap \Hom((w^{U_{2n}})^{\otimes k},
  (w^{U_{2n}})^{\otimes k})$ lies in fact in $\Hom(w^{\otimes k},w^{\otimes
    k})$. By linear independence, $T$ can be uniquely expressed as the sum
  \begin{displaymath}
    T = \sum_{i \in \{1,2\}^k} \sum_{q \in S'(k,k)}\lambda_{q, i}T_{q,i} \qquad 
    (\lambda_{q, i} \in \mathbb C).
  \end{displaymath}
  Denote by $1^k = (1,1,\ldots, 1)$ the constant $k$-tuple.  Since we can
  decompose each $T_q$ as $T_q = \sum_{i \in \{1,2\}^k} T_{q, i}$, we may
  subtract $T_1 := \sum_{q\in S'(k,k)} \lambda_{q, 1^k} T_q \in \Hom(w^{\otimes
    k}, w^{\otimes k})$ from $T$ and consequently assume for the remainder that
  $T P_{1^k} = 0$.

  Now consider an arbitrary $k$-tuple $i$.  We claim that $T|_{H_i} = 0$.  To
  see this, consider the linear map $g : H \to H_1$, $e_r \mapsto e_r$, $e_{r+n}
  \mapsto e_r$ for all $r = 1, \ldots, n$. Observe that the restriction
  $g^{\otimes k} : H_i \to H_{1^k}$ is an isomorphism and that $g^{\otimes k}
  T_{q,i} = T_{q,1^k} g^{\otimes k} P_i$. Moreover, since $T$ is an
  $U_{2n}$-intertwiner it is a linear combination of maps $T_r$, $r \in S(k,k)$,
  and as each of these maps verifies the relation $h^{\otimes k} T = T
  h^{\otimes k}$ for any $h \in \mc B(H)$. We have then
  \begin{displaymath}
    0 = Tg^{\otimes k}P_i = g^{\otimes k} T P_i = 
    g^{\otimes k} \sum_q \lambda_{q,i}T_{q,i} 
    = \sum_q \lambda_{q,i}T_{q,1^k} g^{\otimes k} P_i.
  \end{displaymath}
  Since the family $T_{q,1^k}$, $q\in S'(k,k)$ is linearly independent and $g^{\otimes k}P_i: H_i \to H_{1^k}$ is an isomorphism, we
  conclude that $\lambda_{q,i} = 0$ for each $q$.  Finally, since $H^{\otimes k}
  = \oplus_{i}H_{i}$ and $i$ was arbitrary, this implies $T = 0$.
  I.e., $T = T_1 \in \Hom(w^{\otimes k}, w^{\otimes k})$.
\end{proof}

\section{Applications} \label{section:apps}

\subsection{Free entropy dimension}
In this section we present an application of our hyperlinearity results to the computation
of the free entropy dimension of the canonical generators of $L^\infty(O_N^+)$.  We refer the reader to the survey \cite{Voic} for details on the various notions of free entropy dimension and related concepts.

Let $\Gamma$ be a finitely generated discrete group with a finite symmetric
system of generators $(g_i)_{i=1}^n$, and put $x_i = \Re \lambda(g_i), y_i = \Im
\lambda(g_i) \in \mc L(\Gamma)$.  In \cite[Corollary 4.9]{CoSh}, Connes and
Shlyakhtenko showed that the (non-microstates) free entropy dimension
$\delta^*(x_i, y_i)$ verifies the inequality
\begin{align}\label{CS_ineq}
  \delta^*(x_i, y_i) \le \beta_1^{(2)}(\Gamma)-\beta_0^{(2)}(\Gamma) +1,
\end{align}
where $\beta_k^{(2)}(\Gamma)$ is the $k$th $\ell^2$-Betti number of $\Gamma$. 
On the other hand by \cite{BCG}, the (modified) microstates free entropy
dimension is known to satisfy the inequality
\begin{align}\label{BCG_ineq}\delta_0(x_i,y_i) \le \delta^*(x_i,y_i) .
\end{align} Finally, if $\mc L(\Gamma)$ is diffuse and has the Connes embedding
property, it was shown in \cite[Corollary 4.7]{Jun} that
\begin{align}\label{Jun_ineq}
  1 \le \delta_0(x_i,y_i).
\end{align}
All of the above inequalities apply to the case of quantum groups.  More
precisely, let $\G$ be a compact matrix quantum group of Kac type with diffuse
Connes embeddable von Neumann algebra $L^\infty(\G)$, let $(u^i)_{i=1}^n$ be a
self-conjugate family of inequivalent irreducible unitary representations of
$\G$ whose matrix elements generate $\Pol(\G) \subseteq L^\infty(\G)$, and let
$(x^i_{kl}), (y^i_{kl}) \subset L^\infty(\G)$ be the real and imaginary parts of
the matrix elements of these representations (respectively).  Then we have the
chain of inequalities
\begin{align} \label{ineq:quantum}
1 \le \delta_0(x^i_{kl},y^i_{kl}) \le \delta^*(x^i_{kl},y^i_{kl}) \le
\beta_1^{(2)}(\G)-\beta_0^{(2)}(\G) +1,
\end{align}
where $\beta^{(2)}_k(\G)$ is the \textit{$k$th $\ell^2$-Betti number} of the
compact quantum group $\G$ introduced by Kyed. See \cite{Ky08} and \cite{Lu98}.  Putting all of this together, we obtain the following result.

\begin{theorem}
  The microstates (and non-microstates) free entropy dimension of
  $L^\infty(O_N^+)$ associated to the canonical generators $(u_{ij})_{1 \le i,j
    \le N}$ is $1$ for all $N \ge 4$.
\end{theorem}

\begin{proof}
  It was proved in \cite{cht,Ve} that
  the right hand side of inequality \eqref{ineq:quantum} is exactly $1$. Together
  with the above discussion, the proof is complete.
\end{proof}

\subsection{A remark on the classification of intermediate quantum subgroups
  between $O_N$ and $O_N^+$} \label{section:maximality}

Let $N \ge 3$ and consider the quantum group $O_N^+$.  It is currently an open
problem to determine all intermediate quantum subgroups $O_N \le \G \le O_N^+$.
At this time, only one such intermediate quantum subgroup is known, namely 
the {\it half-liberated orthogonal quantum group} $O_N \le O_N^* \le O_N^+$
\cite{BaSp}.  Moreover, it is known that the inclusion $O_N \le O_N^*$ is
maximal, meaning that there is no intermediate quantum group $O_N \le \G \le
O_N^*$.  See \cite{BaBiCoCu} for details.  It is also conjectured in
\cite{BaBiCoCu} that the inclusion $O_N^* \le O_N^+$ is maximal.  In this short
section we state an easy consequence of Theorem~\ref{theorem:N>3}, which can be regarded as shedding some light on this
conjecture. 

Let $k \ge 1$ and $E \subset \C^N$ a $k$-dimensional subspace.  Generalizing
the case $k=N-1$ discussed in Section \ref{section:mainresult}, we can
canonically associate to $E$ the quantum subgroup $O_{k,E}^+\le O_N^+$
(isomorphic to $O_k^+$ acting on $E$ via its fundamental representation and
acting trivially on $E^\perp$), as well as the corresponding subgroups
$O_{k,E}^*$, $O_{k,E}$.

\begin{theorem}
  If $\G \leq O_N^+$ is a quantum subgroup containing $O_N^*$ and $O_{3,E}^+$,
  then $\G = O_N^+$.
\end{theorem}

\begin{proof}
  From Theorem~\ref{theorem:N>3} it follows by a simple induction on $k = \dim E
  \geq 3$ that $O_N^+$ is generated by $O_N$ and any subgroup $O_{k,E}^+$. In
  particular it is generated by $O_N^*$ and $O_{3,E}^+$.
\end{proof}




\section*{Acknowledgments}

We 
would like to thank Teo Banica, Julien Bichon, Marius Junge and Reiji Tomatsu for enlightening
conversations.  M.B.'s research was partially supported by an NSERC
postdoctoral fellowship. B.C.'s research was partially supported by NSERC, ERA, Kakenhi and
ANR-14-CE25-0003 funding.

\bibliographystyle{plain} \bibliography{CEP4Ao(n)}

\end{document}